\documentclass[runningheads]{llncs}

\let\subparagraph\paragraph
\usepackage{subcaption}
\usepackage{graphicx}%%
\usepackage{array}
\usepackage{url}
\usepackage{times}
\usepackage{amssymb}
\usepackage{amsmath}
\usepackage{latexsym}
\usepackage{graphics}
\usepackage{url}
\usepackage{verbatim}
\usepackage{color}
\usepackage{lineno}
\usepackage{caption}
\usepackage{subcaption}
\usepackage{lineno}
\usepackage{soul}
\usepackage{comment}
\usepackage{datenumber}
\usepackage{amsthm}
\usepackage{paralist}
\usepackage{enumerate}
\usepackage{cases}
\usepackage{thm-restate}% CTAN:macros/experimental/thmtools

\usepackage[noend]{algpseudocode}
\usepackage{algorithm}
\usepackage{cancel}

\newcommand{\BRGC}{BRGC}
\newcommand{\BRGCREV}{\overline{BRGC}}
\newcommand{\FLIP}[2]{flip_{#1}(#2)}
\newcommand{\SWAP}[3]{swap_{#1}(#2, #3)}
\newcommand{\FLIPTWO}[3]{flip2_{#1}(#2, #3)}
\newcommand{\BINARY}[1]{\mathbf{B}(#1)}

\makeatletter
\newcommand{\xxx}[1]{%
  \ifmeasuring@
    #1
  \else
    \sbox0{$\displaystyle#1$}%
    \raisebox{0.8ex}{\rlap{\color{red}%
          \vrule height 0.65pt depth 0.65pt width \wd0}}%
    \box0
  \fi
}

\newcommand{\squishlist}{
 \begin{list}{$\bullet$}
  { \setlength{\itemsep}{0pt}
     \setlength{\parsep}{3pt}
     \setlength{\topsep}{3pt}
     \setlength{\partopsep}{0pt}
     \setlength{\leftmargin}{2.5em}
     \setlength{\labelwidth}{1em}
     \setlength{\labelsep}{0.5em} } }

\newcommand{\squishlisttwo}{
 \begin{list}{$\triangleright$}
  { \setlength{\itemsep}{0pt}
     \setlength{\parsep}{0pt}
    \setlength{\topsep}{0pt}
    \setlength{\partopsep}{0pt}
    \setlength{\leftmargin}{2em}
    \setlength{\labelwidth}{1.5em}
    \setlength{\labelsep}{0.5em} } }

\newcommand{\squishend}{
  \end{list}  }

\usepackage{listings}

\definecolor{verbgray}{gray}{0.9}
\lstnewenvironment{code}{%
  \lstset{backgroundcolor=\color{white},
 frame=single,
  language=C,
  framerule=0pt,
  basicstyle=\ttfamily,
  showstringspaces=false,
  commentstyle=\color{blue}\textit,
  keywordstyle=\color{black}\bf,
  numberstyle=\footnotesize,
  breaklines=true,
  columns=fullflexible}}{}

\definecolor{shadecolor}{rgb}{.9, .9, .9}

\algdef{SE}[DOWHILE]{Do}{doWhile}{\algorithmicdo}[1]{\algorithmicwhile\ #1}%

\graphicspath{{graphics/}}
\begin{document}

\title{Inside the Binary Reflected Gray Code: \\ Flip-Swap Languages in 2-Gray Code Order}
%\title{Generating cyclic $2$-Gray codes in reflected Gray code order}
\date{}
\author{
Joe Sawada
%\thanks{School of Computer Science, University of Guelph, Canada. Research supported by NSERC. \texttt{ email:jsawada@uoguelph.ca}}
\ \ \   \
Aaron Williams
%\thanks{  \texttt{ email:  }}
 \ \ \  \
Dennis Wong %\thanks{  \texttt{ email:  }}
\institute{}
}
\maketitle

\begin{abstract}
A \emph{flip-swap language} is a set ${\bf S}$ of binary strings of length $n$ such that ${\bf S} \cup \{0^n\}$ is closed under two operations (when applicable):  (1) Flip the leftmost $1$; and  (2) Swap the leftmost $1$ with the bit to its right.
%Natural representations of many combinatorial objects are flip-swap languages including:
Flip-swap languages model many combinatorial objects including
necklaces,  Lyndon words, prefix normal words, left factors of $k$-ary Dyck words, and feasible solutions to $0$-$1$ knapsack problems.
We prove that any flip-swap language forms a cyclic $2$-Gray code when listed in binary reflected Gray code (BRGC) order.
Furthermore, a generic successor rule computes the next string when provided with a membership tester. % to compute the next string.
The rule generates each string in the aforementioned flip-swap languages in $O(n)$-amortized per string, except for prefix normal words of length $n$ which require $O(n^{1.864})$-amortized per string.
Our work generalizes results on necklaces and Lyndon words by Vajnovski [Inf. Process. Lett. 106(3):96$-$99, 2008].
\end{abstract}

%\section{Successor rule for necklaces in lexicographic order} \label{sec:gen_n-lex}
\section{Introduction} \label{sec:intro}
Combinatorial generation studies the efficient generation of each instance of a combinatorial object, such as the $n!$ permutations of $\{1,2,\ldots,n\}$ or the $\frac{1}{n+1}\binom{2n}{n}$ binary trees with $n$ nodes.
The research area is fundamental to computer science and it has been covered by textbooks such as \emph{Combinatorial Algorithms for Computers and Calculators} by Nijenhuis and Wilf \cite{WilfBook}, \emph{Concrete Mathematics: A Foundation for Computer Science} by Graham, Knuth, and Patashnik~\cite{concrete}, and \emph{The Art of Computer Programming, Volume 4A, Combinatorial Algorithms} by Knuth~\cite{opac-b1101743}.
In fact, Knuth's section on \emph{Generating Basic Combinatorial Patterns} is over 450 pages.
The subject is important to every day programmers, and Arndt's \emph{Matters Computational: Ideas, Algorithms, Source Code} is an excellent practical resource \cite{MattersComputational}.
A primary consideration is listing the instances of a combinatorial object so that consecutive instances differ by a specified \emph{closeness condition}. % involving a constant amount of change.
Lists of this type are called \emph{Gray codes}.
This terminology is due to the eponymous \emph{binary reflected Gray code} (\emph{BRGC}) by Frank Gray, which orders the $2^n$ binary strings of length $n$ so that consecutive strings differ in one bit.
The BRGC was patented for a pulse code communication system in 1953~\cite{gray-pulse-code-communication-1953}.
For example, the order for $n=4$ is
\begin{equation} \label{eq:BRGC4}
\begin{aligned}
0000, 1000, 1100, 0100, 0110, 1110, 1010, 0010, \\ 0011, 1011, 1111, 0111, 0101, 1101, 1001, 0001.
\end{aligned}
\end{equation}
Variations that reverse the entire order or the individual strings are also commonly used in practice and in the literature.
We note that the order in \eqref{eq:BRGC4} is \emph{cyclic} because the last and first strings also differ by the closeness condition, and this property holds for all $n$.

One challenge facing combinatorial generation is its relative surplus of breadth and lack of depth\footnote{This is not to say that combinatorial generation is always easy.
For example, the `middle levels` conjecture was confirmed by M\"{u}tze \cite{middleLMS} after 30 years and effort by hundreds of researchers.}.
%As a result, practicioners must learn a variety of results, and it is telling that
%It is also telling that
For example, \cite{MattersComputational}, \cite{opac-b1101743}, and \cite{WilfBook} have separate subsections for different combinatorial objects, and the majority of the Gray codes are developed from first principles.
Thus, it is important to encourage simple frameworks that can be applied to a variety of combinatorial objects.
Previous work in this direction includes the following:
\begin{enumerate}
\item the ECO framework developed by Bacchelli, Barcucci, Grazzini, and Pergola~\cite{Bacchelli2004} that generates Gray codes for a variety of combinatorial objects such as Dyck words in constant amortized time per instance;
\item the twisted lexico computation tree by Takaoka~\cite{DBLP:conf/isaac/Takaoka99} that generates Gray codes for multiple combinatorial objects in constant amortized time per instance; %including multiset permutations,
\item loopless algorithms developed by Walsh~\cite{DBLP:conf/dmtcs/Walsh03} to generate Gray codes for multiple combinatorial objects, which extend algorithms initially given by Ehrlich in~\cite{Ehrlich:1973:LAG:321765.321781};
\item greedy algorithms observed by Williams~\cite{GreedyWADS} that provide a uniform understanding for many previous published results;
\item the reflectable language framework by Li and Sawada~\cite{Li2009296} for generating Gray codes of $k$-ary strings, restricted growth strings, and $k$-ary trees with $n$ nodes; %internal
\item the bubble language framework developed by Ruskey, Sawada and Williams~\cite{Ruskey2012155} that provides algorithms to generate shift Gray codes for fixed-weight necklaces and Lyndon words, $k$-ary Dyck words, and representations of interval graphs;
\item the permutation language framework developed by Hartung, Hoang, M\"{u}tze and Williams~\cite{10.5555/3381089.3381163} that provides algorithms to generate Gray codes for a variety of combinatorial objects based on encoding them as permutations.
\end{enumerate}

We focus on an approach that is arguably simpler than all of the above: % from a conceptual point of view.
Start with a known Gray code and then \emph{filter} or \emph{induce} the list based on a subset of interest.
In other words, the subset is listed in the relative order given by a larger Gray code, and the resulting order is a \emph{sublist (Gray code)} with respect to it.
Historically, the first sublist Gray code appears to be the \emph{revolving door} Gray code for combinations \cite{Wilf}.
A \emph{combination} is a length $n$ binary string with \emph{weight} (i.e. number of ones) $k$.
The Gray code is created by filtering the BRGC, as shown below for $n=4$ and $k=2$ (cf. \eqref{eq:BRGC4})
\begin{equation} \label{eq:RevolvingDoor42}
\begin{aligned}
\cancel{0000}, \cancel{1000}, 1100, \cancel{0100}, 0110, \cancel{1110}, 1010, \cancel{0010}, \\ 0011, \cancel{1011}, \cancel{1111}, \cancel{0111}, 0101, \cancel{1101}, 1001, \cancel{0001}.
\end{aligned}
\end{equation}
This order is a \emph{transposition Gray code} as consecutive strings differ by transposing two bits\footnote{When each string is viewed as the incidence vector of a $k$-subset of $\{1,2,\ldots,n\}$, then consecutive $k$-subsets change via a ``revolving door'' (i.e. one value enters and one value exits).}.
It can be generated \emph{directly} (i.e. without filtering) by an efficient algorithm~\cite{Wilf}.
Transposition Gray codes are a special case of \emph{2-Gray codes} where consecutive strings differ by flipping (i.e. complementing) at most two bits.
Vajnovszki~\cite{vaj} proved that necklaces and Lyndon words form a cyclic $2$-Gray code in BRGC order, and efficient algorithms can generate these sublist Gray codes directly \cite{neck-sww}.
Our goal is to  expand upon the known languages that are 2-Gray codes in BRGC order, and which can be efficiently generated.
To do this, we introduce a new class of languages.

A \emph{flip-swap language} (with respect to 1) is a set ${\bf S}$ of length $n$ binary strings such that ${\bf S} \cup \{0^n\}$ is closed under two operations (when applicable):
(1) Flip the leftmost $1$; and
(2) Swap the leftmost $1$ with the bit to its right.
A flip-swap language with respect to $0$ is defined similarly. %by changing $0$ and $1$.
Flip-swap languages encode a wide variety of combinatorial objects.
%The formal definitions of these languages are given in Section~\ref{sec:flip-swap}.

%===============================
\begin{theorem} \label{thm:main1}
The following sets of length $n$ binary strings are flip-swap languages:

\begin{tabular}{p{0.51\textwidth}p{0.49\textwidth}}
\small
{\bf Flip-Swap languages (with respect to $1$)}
\squishlisttwo
\item[i.]     all strings
\item[ii.]    strings with weight $\leq k$
\item[iii.]   strings $\leq \gamma$
\item[iv.]    strings with $\leq k$ inversions re: $0^*1^*$ % apropos means "with respect to"
\item[v.]     strings with $\leq k$ transpositions re: $0^*1^*$
\item[vi.]    strings $<$ their reversal
\item[vii.]   strings $\leq$ their reversal (neckties)
\item[viii.]  strings $<$ their complemented reversal
\item[ix.]    strings $\leq$ their complemented reversal
\item[x.]     strings with forbidden $10^t$
\item[xi.]    strings with forbidden prefix $1\gamma$
\item[xii.]   $0$-prefix normal words
\item[xiii.]  necklaces (smallest rotation)
\item[xiv.]   Lyndon words
\item[xv.]    prenecklaces (smallest rotation)
\item[xvi.]   pseudo-necklaces with respect to $0^*1^*$
\item[xvii.]  left factors of $k$-ary Dyck words %prefixes of balanced parentheses,
\item[xviii.] feasible solutions to 0-1 knapsack problems %, and
%\item union, intersection and quotient of the above sets.
\squishend
&
\small
{\bf Flip-Swap languages (with respect to~$0$)}
\squishlisttwo
\item all strings
\item strings with weight $\geq k$
\item strings $\geq \gamma$
\item strings with $\leq k$ inversions re: $1^*0^*$
\item strings with $\leq k$ transpositions re: $1^*0^*$
\item strings $>$ their reversal
\item strings $\geq$ their reversal
\item strings $>$ their complemented reversal
\item strings $\geq$ their complemented reversal
\item strings with forbidden $01^t$
\item strings with forbidden prefix $0\gamma$
\item $1$-prefix normal words
\item necklaces (largest rotation)
\item aperiodic necklaces (largest rotation)
\item prenecklaces (largest rotation)
\item pseudo-necklaces with respect to $1^*0^*$
\squishend
\end{tabular}
\end{theorem}
%===============================

% Later in this paper we will demonstrate that the flip-swap languages including $0^n$ correspond to the ideals of a partially ordered set. These include all the languages in Theorem~\ref{thm:main1}   except for Lyndon words (aperiodic necklaces) and binary strings that are strictly less (greater) then their reversals.

Our second result is that every flip-swap language forms a cyclic $2$-Gray code when listed in BRGC order. % (with respect to 0 or 1 accordingly)
This generalizes the previous sublist BRGC results \cite{neck-sww,vaj}.

%===============================
\begin{theorem} \label{thm:main2}
When a flip-swap language $\mathbf{S}$ is listed in BRGC order the resulting listing is a 2-Gray code.  If $\mathbf{S}$ includes $0^n$ then the listing is cyclic.
\end{theorem}
%===============================

Our third result is a generic successor rule, which efficiently computes the next string in the $2$-Gray code of a flip-swap language, so long as a fast membership test is given.
%This leads to the following.

%===============================
\begin{theorem} \label{thm:main3}
The languages in Theorem~\ref{thm:main1} can be generated in $O(n)$-amortized time per string, with the exception of prefix normal words which require $O(n^{1.864})$-time. % per string.
\end{theorem}
%==============================

%The remainder of this paper is outlined as follows.
In Section~\ref{sec:BRGC}, we formally define our version of the BRGC.
%version of the BRGC that we use.
In Section~\ref{sec:flip-swap}, we prove Theorem~\ref{thm:main1}, and define the flip-swap partially ordered set.
% and define the languages listed in its statement.
%In Section~\ref{sec:poset}, we introduce the flip-swap poset.
In Section~\ref{sec:main-result}, we give our generic successor rule and prove Theorem~\ref{thm:main2}.
In Section~\ref{sec:gen_n}, we present a generic generation algorithm that list out each string of a flip-swap language, and we prove Theorem \ref{thm:main3}.
%A C program that generates each of our sublist Gray codes is available on the Combinatorial Object Server~\cite{cos} and is discussed in Section~\ref{sec:gen_n}.

%================================================
%================================================
%================================================
%================================================
\section{The Binary Reflected Gray Code} \label{sec:BRGC}

%The \emph{binary reflected Gray code} (\emph{BRGC}) was named after Frank Gray and patented for a pulse code communication system in 1953~\cite{gray-pulse-code-communication-1953}.
%Let ${\bf B}(n)$ denote the set of binary strings of length $n$.
Let $\BINARY{n}$ denote the set of length $n$ binary strings.
Let $\BRGC(n)$ denote the listing of $\BINARY{n}$ in BRGC order.
Let $\BRGCREV(n)$ denote the listing $\BRGC(n)$ in reverse order.
Then $\BRGC(n)$ can be defined recursively as follows, where $\mathcal{L} \cdot x$ denotes the listing $\mathcal{L}$ with the character $x$ appended to the end of each string:
\begin{equation*}
\BRGC(n)  =
\begin{cases}
0,1   & \text{if $n=1$};   \\
 \BRGC(n-1) \cdot 0, \   \BRGCREV(n-1) \cdot 1  & \text{if $n >1$.}
\end{cases}
\end{equation*}
For example,  $\BRGC(2) = 00, 10, 11, 01$ and $\BRGCREV(2) = 01, 11, 10, 00$, {thus}
\begin{center}
$\BRGC(3) = 00{\bf 0}, 10{\bf 0}, 11{\bf 0}, 01{\bf 0}, 01{\bf 1}, 11{\bf 1}, 10{\bf 1}, 00{\bf 1}$.
\end{center}
This definition of BRGC order is the same as the one used by Vajnovzski~\cite{vaj}.
When the strings are read from right-to-left, we obtain the classic definition of BRGC order~\cite{gray-pulse-code-communication-1953}.  For flip-swap languages with respect to 0, we interchange the roles of the 0s and 1s; however, for our  discussions we will focus on flip-swap languages with respect to 1.
Table \ref{table:brgc} illustrates $\BRGC(4)$ and six flip-swap languages listed in Theorem~\ref{thm:main1}.

\begin{table}[t]
    \captionsetup{format=hang}
    \centering
    \begin{subfigure}[b]{0.5\textwidth}
        \centering
        \scriptsize
        \begin{tabular}{ | c | c | c | c | c | c | c |  } \hline
        %$\BINARY{4}$  & Necklaces & $0$-PNW  &  $\BINARY{4} \leq 1001$  & $\BINARY{4}$ with weight $\leq 2$ & $\BINARY{4} \leq$ its reversal \\ \hline
        $n=4$ & \ all \ & necklaces & $0$-PNW & $\leq 1001$ & $k \leq 2$ & neckties \\
        BRGC & i.  & xiii. & xii.  &  iii. & ii. & vii. \\
        \hline

          0000 & \checkmark &\checkmark    & \checkmark  &  \checkmark & \checkmark & \checkmark  \\
          1000 & \checkmark &    &   &  \checkmark &\checkmark &    \\
          1100 & \checkmark &    &   &   &\checkmark &    \\
          0100 & \checkmark &    &   & \checkmark  & \checkmark &    \\
          0110 & \checkmark &    & \checkmark  &  \checkmark & \checkmark & \checkmark \\
          1110 & \checkmark &    &   &   & &    \\
          1010 & \checkmark &   &   &   & \checkmark &    \\
          0010 & \checkmark &    & \checkmark  &  \checkmark  & \checkmark & \checkmark \\
          0011 & \checkmark & \checkmark   & \checkmark  &  \checkmark & \checkmark & \checkmark  \\
          1011 & \checkmark &    &   &   & & \checkmark   \\
          1111 & \checkmark & \checkmark   &   &   &  & \checkmark  \\
          0111 & \checkmark & \checkmark   & \checkmark  &  \checkmark &  & \checkmark  \\
          0101 & \checkmark & \checkmark  & \checkmark  &  \checkmark & \checkmark & \checkmark  \\
          1101 & \checkmark &   &   &   & &    \\
          1001 & \checkmark &   &   &  \checkmark & \checkmark & \checkmark  \\
          0001 & \checkmark & \checkmark   & \checkmark  &  \checkmark & \checkmark & \checkmark  \\
          \hline
         \end{tabular}
         \caption{String membership in 6 flip-swap languages.}
         \label{table:brgc_checkmarks}
    \end{subfigure}
    \begin{subfigure}[b]{0.494\textwidth}
        \centering
        \scriptsize
        \begin{tabular}{ c c c c c c  }
        %$\BINARY{4}$  & Necklaces & $0$-PNW  &  $\BINARY{4} \leq 1001$  & $\BINARY{4}$ with weight $\leq 2$ & $\BINARY{4} \leq$ its reversal \\ \hline
        i.  & xiii. & xii.  &  iii. & ii. & vii. \\[-0.7em]
        \ \scalebox{0.75}[1.0]{\raisebox{-0.2em}{\includegraphics[angle=270,scale=2.00,trim=1 1 1 1]{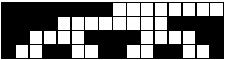}}}  \ &
        \ \scalebox{0.75}[1.0]{\raisebox{-0.2em}{\includegraphics[angle=270,scale=2.00,trim=1 1 1 1]{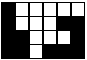}}} \ &
        \ \scalebox{0.75}[1.0]{\raisebox{-0.2em}{\includegraphics[angle=270,scale=2.00,trim=1 1 1 1]{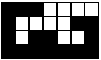}}} \ &
        \ \scalebox{0.75}[1.0]{\raisebox{-0.3em}{\includegraphics[angle=270,scale=2.00,trim=1 1 1 1]{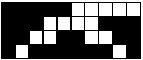}}} \ &
        \ \scalebox{0.75}[1.0]{\raisebox{-0.3em}{\includegraphics[angle=270,scale=2.00,trim=1 1 1 1]{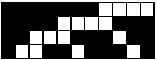}}} \ &
        \ \scalebox{0.75}[1.0]{\raisebox{-0.3em}{\includegraphics[angle=270,scale=2.00,trim=1 1 1 1]{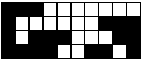}}}
        \end{tabular}
        \caption{Visualizating the 2-Gray codes in (a).}
        \label{table:brgc_checkmarks}
    \end{subfigure}
    \caption{Flip-swap languages ordered as sublists of the binary reflected Gray code.
    Theorem \ref{thm:main1} covers each language, so the resulting orders are 2-Gray codes.}
    \label{table:brgc}
\end{table}
%\caption{\small Six flip-swap languages with respect to $1$: ${\bf B}(4)$, necklaces, $0$-prefix normal words, binary strings that are lexicographically smaller than or equal to $1001$, binary strings with weight less than or equal to $2$, and binary strings that are lexicographically smaller or equal to its reversal of length $4$ in BRGC order. }
%\vspace{-30}

%====================================
%====================================
%====================================
%====================================
\section{Flip-swap languages} \label{sec:flip-swap}

In this section, we formalize some of the non-obvious flip-swap languages stated in Theorem~\ref{thm:main1}. Then we prove  Theorem~\ref{thm:main1} for a subset of the listed languages including necklaces, prefix normal words, and feasible solutions to the 0-1 knapsack problems.  The remainder of the languages are proved in the Appendix.

Consider a binary string $\alpha = b_1b_2\cdots b_n$.
The \emph{weight} of $\alpha$ is the  number of 1s it contains.
An \emph{inversion} in $\alpha$ with respect to $0^*1^*$ is an index pair $(i,j)$ such  that $i<j$ and $b_i = 1$ and $b_j = 0$.
The number of  \emph{transpositions}  of $\alpha$ with respect to another binary string $\beta$ of length $n$ is the minimum number of adjacent transpositions required to transform $\alpha$ to $\beta$.

A \emph{necklace} is the lexicographically smallest (largest) string in an equivalence class under rotation.
An \emph{aperiodic necklace} is a necklace that cannot be written in the form $\beta^j$ for some $j < n$.   A \emph{Lyndon word} is an aperiodic necklace  when using the lexicographically smallest string as the representative.
A \emph{prenecklace} is a prefix of a necklace.
A \emph{block} with respect to $0^*1^*$ is a maximal substring of the form $0^*1^*$.
A string $\alpha = b_1b_2 \cdots b_n = B_b B_{b-1} \cdots B_1$ is a \emph{pseudo-necklace} with respect to $0^*1^*$ if $B_b \leq B_i$ for all $1 \leq  i < b$.

%A binary string $\alpha$ is \emph{prefix normal} with respect to $0$ (also known as $0$-prefix normal word) if no substring of $\alpha$ has more $0$s than its prefix of the same length.

%
A \emph{$k$-ary Dyck word} is a binary string of length $n = tk$ with $t$ copies of $1$ and $t(k - 1)$ copies of $0$ such that every prefix has at most $k - 1$ copies of $0$ for every $1$.
The set of length $n$ prefixes of $k$-ary Dyck words is called \emph{left factors of $k$-ary Dyck words}.
%A \emph{left factor} of a $k$-ary Dyck words is...
%

%The input to a \emph{$0$-$1$ knapsack problem} is a knapsack capacity $W$, and a set of $n$ items each of which has a non-negative weight $w_i \geq 0$ and a value $v_i$.
%A subset of items is \emph{feasible} if the total weight of the items in the subset is less than or equal to the capacity $W$.
%Typically, the goal of the problem is to find a feasible subset with the maximum value, or to decide if a feasible subset exists with value $\geq c$.

Let $\FLIP{\alpha}{i}$ be the string obtained by complementing $b_i$.
Let $\SWAP{\alpha}{i}{j}$ be the string obtained by swapping $b_i$ and $b_j$.
When the context is clear we use $flip(i)$ and $swap(i, j)$ instead of $\FLIP{\alpha}{i}$ and $\SWAP{\alpha}{i}{j}$.
%Let $w(\alpha)$ denote the weight of $\alpha$.
Also, let $\ell_0(\alpha)$ denote the position of the leftmost $0$ of $\alpha$ or $n+1$ if no such position exists.
Similarly, let $\ell_1(\alpha)$ denote the position of the leftmost $1$ of $\alpha$ or $n+1$ if no such position exists.
We now prove that binary strings, necklaces, prefix normal words,  and feasible solutions to the 0-1 knapsack problems are flip-swap languages with respect to 1.  %

{\bf Binary strings}: Obviously the set ${\bf B}(n)$ satisfies the two closure properties of a flip-swap language and thus is a flip-swap language.
In fact, the BRGC order induces a cyclic $1$-Gray code for ${\bf B}(n)$~\cite{opac-b1101743,ruskey}.

{\bf Necklaces}:
%A \emph{necklace} is the lexicographically smallest string in equivalence class of strings under rotation.
Let ${\bf N}(n)$ be the set of necklaces of length $n$ and $\alpha = 0^j 1 b_{j+2} b_{j+3}\cdots b_n$ be a necklace in ${\bf N}(n)$. % with $\alpha \neq 0^{n-1}1$.
By the definition of necklace, it is easy to see that $\FLIP{\alpha}{\ell_\alpha} = 0^{j+1} b_{j+2} b_{j+3}\cdots b_n \in {\bf N}(n)$ and thus ${\bf N}(n)$ satisfies the flip-first property.
For the swap-first operation, observe that if $\alpha \neq 0^{n-1}1$ and $b_{j+2} = 1$, then the swap-first operation produces the same necklace.
Otherwise if $\alpha \neq 0^{n-1}1$ and $b_{j+2} = 0$, then the swap-first operation produces the string $0^{j+1} 1 b_{j+3} b_{j+4}\cdots b_n$ which is clearly a necklace.
Thus, the set of necklaces is a flip-swap language.

{\bf Prefix normal words}:
A binary string $\alpha$ is \emph{prefix normal} with respect to $0$ (also known as $0$-prefix normal word) if no substring of $\alpha$ has more $0$s than its prefix of the same length.
For example, the string 001010010111011 is a $0$-prefix normal word but the string 001010010011011 is not because it has a substring of length $5$ with four $0$s while the prefix of length $5$ has only three $0$s.
%There has been much interest recently on prefix normal words which have applications in binary jumbled pattern matching~\cite{DBLP:conf/fun/BurcsiFLRS14,DBLP:conf/cpm/BurcsiFLRS14,DBLP:journals/tcs/BurcsiFLRS17,DBLP:conf/dlt/FiciL11,pnw-lex}.

Observe that the set of $0$-prefix normal words of length $n$ satisfies the two closure properties of a flip-swap language as the flip-first and swap-first operations either increases or maintain the number of $0$s in its prefix. %, while the swap-first operation also decreases the number of $0$s of its suffix.
Thus, the set of $0$-prefix normal words of length $n$ is a flip-swap language.

{\bf Feasible solutions to $0$-$1$ knapsack problems}:
The input to a \emph{$0$-$1$ knapsack problem} is a knapsack capacity $W$, and a set of $n$ items each of which has a non-negative weight $w_i \geq 0$ and a value $v_i$.
A subset of items is \emph{feasible} if the total weight of the items in the subset is less than or equal to the capacity $W$.
Typically, the goal of the problem is to find a feasible subset with the maximum value, or to decide if a feasible subset exists with value $\geq c$.

Given the input to a $0$-$1$ knapsack problem, we reorder the items by non-decreasing weight.
That is, $w_i \geq w_{i+1}$ for $1 \leq i \leq n-1$.
Notice that the incidence vectors of feasible subsets are now a flip-swap language.
More specifically, flipping any $1$ to $0$ causes the subset sum to decrease, and so does swapping any $1$ with the bit to its right.
Hence, the language satisfies the flip-first and the swap-first closure properties and is a flip-swap language. %over ${\bf B}(n)$.

%====================================
%====================================
%====================================
%====================================
\subsection{Flip-Swap poset} \label{sec:poset}

In this section we introduce a poset whose ideals correspond to a flip-swap language which includes the string $0^n$.

Let $\alpha = b_1 b_2 \cdots b_n $ be a length $n$ binary string.   We define $\tau(\alpha)$ as follows:
    \begin{subnumcases}{\tau(\alpha) = }
    \alpha  &  if $\alpha = 0^n$,    \nonumber \\
    \FLIP{\alpha}{\ell_\alpha}       &  if  $\alpha \ne 0^n$ and ($\ell_\alpha = n$ or $b_{\ell_\alpha + 1} = 1$) \ \ \ \ \ \ \hfill (flip-first),  \nonumber \\
    \SWAP{\alpha}{\ell_\alpha}{\ell_\alpha+1}     &  otherwise\hfill (swap-first).   \nonumber
    \end{subnumcases}

Let $\tau^t(\alpha)$ denote the string that results from applying the $\tau$ operation $t$ times to $\alpha$.
%We then define a binary relation $<_R$ on ${\bf B}(n)$ such that $\beta <_R \alpha$ if $\beta = \tau^t(\alpha)$ for some $t > 0$.
We  define the binary relation $<_R$ on ${\bf B}(n)$ to be the transitive closure of the cover relation $\tau$, that is $\beta <_R \alpha$ if $\beta \ne \alpha$ and $\beta = \tau^t(\alpha)$ for some $t > 0$.
It is easy to see that the binary relation $<_R$ is irreflexive, anti-symmetric and transitive. Thus $<_R$ is a strict partial order.
The  relation $<_R$ on binary strings defines our flip-swap poset.
%We thus use the binary relation $<_R$ to define our flip-swap poset.

\begin{definition}
The \emph{flip-swap poset} $\mathcal{P}(n)$ is a strict poset with ${\bf B}(n)$ as the ground set and $<_R$ as the strict partial order. %of $\mathcal{P}(n)$.  %and $\alpha \prec \gamma$.
\end{definition}

Figure~\ref{fig:poset} shows the Hasse diagram of $\mathcal{P}(4)$  with the ideal for binary strings of length $4$ that are lexicographically smaller or equal to $1001$ in bold.
Observe that $\mathcal{P}(n)$ is always a tree with $0^n$ as the unique minimum element, and that its ideals are the subtrees that contain this minimum.

\begin{lemma}\label{lem:2BRGC-poset}
A set $ {\bf S} $ over ${\bf B}(n)$ that includes $0^n$ is a flip-swap language if and only if $ {\bf S} $ is an ideal of  $\mathcal{P}(n)$.
\end{lemma}

\begin{proof}
Let $ {\bf S} $ be a flip-swap language over ${\bf B}(n)$ and $\alpha$ be a string in $ {\bf S}$.
Since $ {\bf S} $ is a flip-swap language, $ {\bf S} $ satisfies the flip-first and swap-first properties and thus $\tau(\alpha)$ is a string in $ {\bf S} $.
Therefore every string $\gamma <_R \alpha$ is in $ {\bf S} $ % because of the flip-first and swap-first properties.
and hence ${\bf S} $ is an ideal of  $\mathcal{P}(n)$.
The other direction is similar.
\end{proof}

If {\bf S} is a set of binary strings and $\gamma$ is a binary string, then the \emph{quotient} of {\bf S} and $\gamma$ is ${\bf S}/\gamma = \{\alpha \ | \ \alpha  \gamma \in {\bf S}\}$.
%We now prove two results for flip-swap languages.% and the flip-swap poset.

\begin{figure}[t]
    \captionsetup{format=hang}
    \centering
    \begin{subfigure}[b]{0.54\textwidth}
        \centering
        \includegraphics[height=1.1in]{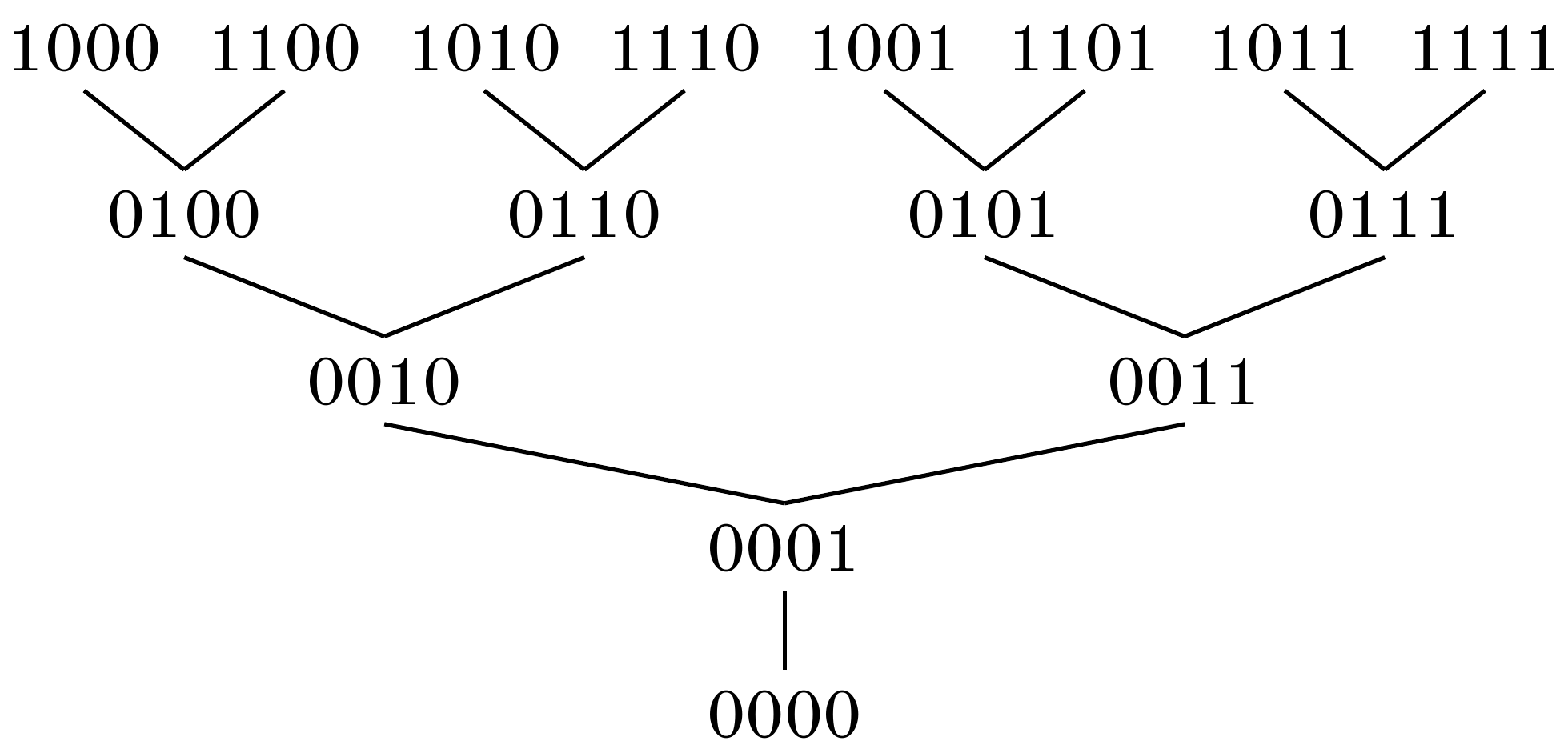}
        \caption{The flip-swap poset $\mathcal{P}(4)$.}
        \label{fig:poset_P4}
    \end{subfigure}
    %\vspace{-15}
    \begin{subfigure}[b]{0.45\textwidth}
        \centering
        \includegraphics[height=1.1in]{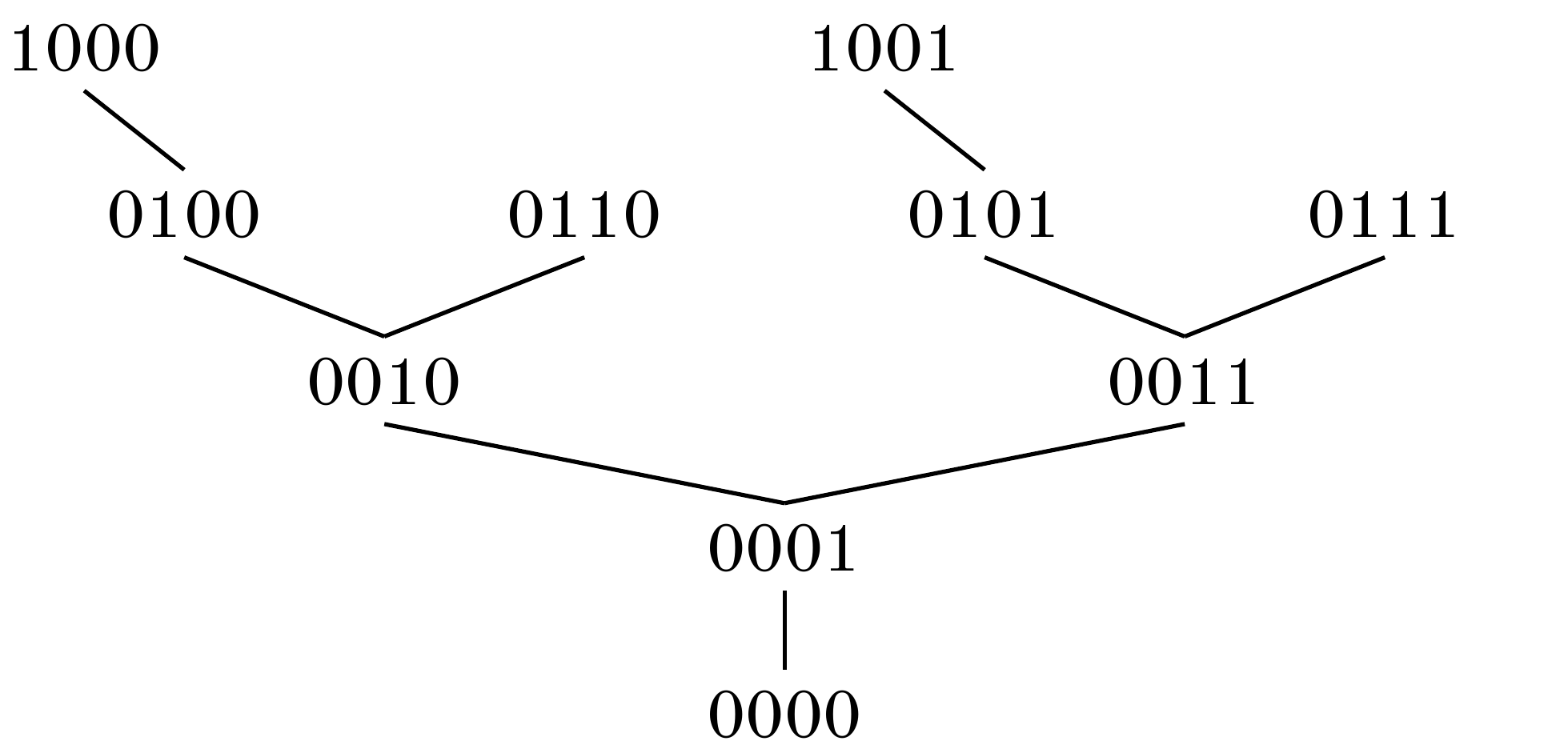}
        \caption{An ideal of $\mathcal{P}(4)$.}
        \label{fig:poset_leq}
    \end{subfigure}
    \caption{Flip-swap languages are the ideals of the flip-swap poset.
    The ideal in (b) contains the $4$-bit binary strings that are $\leq 1001$ with respect to  lexicographic order.}
    \label{fig:poset}
    %\vspace{-10}
\end{figure}

\begin{lemma}\label{lem:closure}
If ${\bf S}_1$ and ${\bf S}_2$ are flip-swap languages and $\gamma$ is a binary string, then ${\bf S}_1 \cap {\bf S}_2$, ${\bf S}_1 \cup {\bf S}_2$ and ${\bf S}_1 / \gamma$ are flip-swap languages.
\end{lemma}

\begin{proof}
Let ${\bf S}_1$ and ${\bf S}_2$ be two flip-swap languages and let $\gamma$ be a binary string.
The intersection and union of ideals of any poset are also ideals of that poset, so ${\bf S}_1 \cap {\bf S}_2$ and ${\bf S}_1 \cup {\bf S}_2$ are flip-swap languages.
Now consider $\alpha \in {\bf S}_1 / \gamma$.

Suppose  $\alpha \in {\bf S}_1 / \gamma$  for some non-empty $\gamma$ where $j = |\alpha|$.  This means that $\alpha \gamma \in {\bf S}_1$.
Consider three cases depending  $\ell_{\alpha\gamma}$.
If $\ell_{\alpha\gamma} < j$, then clearly $\tau(\alpha \gamma)  = \tau(\alpha) \gamma$. From  Lemma~\ref{lem:2BRGC-poset}, $\tau(\alpha) \gamma \in  {\bf S}_1$ and thus $\tau(\alpha) \in {\bf S}_1 / \gamma$.
If $\ell_{\alpha\gamma} = j$, then $\alpha = 0^{j-1}1$ and $\tau(\alpha) = 0^j$. Since ${\bf S}_1$ is a flip-swap language
$0^j\gamma \in {\bf S}_1$.  Again this implies that $\tau(\alpha) \in {\bf S}_1 / \gamma$.    If $\ell_{\alpha\gamma} > j$ then
$\alpha = 0^j$ and $\tau(\alpha) = \alpha$ in this case.   For each case we have shown that  $\tau(\alpha) \in {\bf S}_1 / \gamma$ and thus ${\bf S}_1 / \gamma$ is a flip-swap language by Lemma~\ref{lem:2BRGC-poset}.
\end{proof}

\begin{corollary}\label{lem:union-intersection-quotient}
%The union, intersection and quotient of  flip-swap languages are also flip-swap languages.
Flip-swap languages are closed under union, intersection, and quotient.
\end{corollary}

\begin{proof}
Let ${\bf S}_A$ and ${\bf S}_B$ be flip-swap languages and $\gamma$ be a binary string.
Since ${\bf S}_A$ and ${\bf S}_B$ can be represented by ideals of the flip-swap poset, possibly excluding $0^n$, by Lemma~\ref{lem:closure} the sets ${\bf S}_A \cap {\bf S}_B$, ${\bf S}_A \cup {\bf S}_B$ and ${\bf S}_A/\gamma$ are flip-swap languages.
%The proof for the union, intersection and quotient of near flip-swap languages is similar.
\end{proof}

\begin{lemma}\label{lem:prefix}
%If ${\bf S}$ is a flip-swap language, $\gamma$ is a binary string and the set ${\bf S}/\gamma$ is not empty, then $0^t \gamma \in {\bf S}$ for some $t\geq 0$.
If  $\alpha \gamma$ is a binary string in a flip-swap language ${\bf S}$, then $0^{|\alpha|} \gamma \in {\bf S}$.
\end{lemma}

\begin{proof}
This result follows from the flip-first property of flip-swap languages.
%Suppose there is a string $\alpha \neq 0^n$ in ${\bf S}$ with suffix $\gamma$, then by repeatedly applying $\tau$ on $\alpha$ we obtain the string $0^t \gamma$ for some $t \geq 0$.
\end{proof}

%====================================
%====================================
%====================================
%====================================
\section{A generic successor rule for flip-swap languages}\label{sec:main-result}

Consider any flip-swap language ${\bf S}$ that includes the string $0^n$.
Let $\mathcal{BRGC}({\bf S})$ denote the listing of $\bf S$ in BRGC order.  Given a string $\alpha \in \mathbf{S}$,
we define a generic \emph{successor rule} that computes the string following $\alpha$ in the cyclic listing $\mathcal{BRGC}({\bf S})$.
%The discussion of the rest of the section will be based on flip-swap languages with respect to $1$. Similar results follow for flip-swap languages with respect to $0$.

Let  $\alpha = b_1 b_2 \cdots b_n$ be a string in ${\bf S}$.
Let $t_\alpha$ be the leftmost position such that $\FLIP{\alpha}{t_\alpha} \in {\bf S}$ when $|{\bf S}|>1$, such a $t_\alpha$ exists since ${\bf S}$ satisfies the flip-first property and $|{\bf S}|>1$.
Recall that $\ell_\alpha$ is defined to be the position of the leftmost $1$ of $\alpha$ (or $|\alpha|+1$ if no such position exists).
Notice that $t_\alpha \leq \ell_\alpha$ when $|{\bf S}|>1$ since ${\bf S}$ is a flip-swap language.

\begin{table}[t]
\begin{center} \footnotesize
\begin{tabular}{c @{\hskip 0.2in}  c @{\hskip 0.2in} c @{\hskip 0.2in}  c @{\hskip 0.2in} c @{\hskip 0.2in} c}
%\multicolumn{4}{c}{{\bf Necklaces of length 6 in BRGC order}} \\
Necklaces  &Parity of $w(\alpha)$	& $t_\alpha$  & $\ell_\alpha$  &  Successor& Case\\ \hline
000000 & even &$6$ &  & $flip2(5, 6)$  &(\ref{f3})   \\
000011 & even& $3$ &   & $flip2(2, 3)$ &(\ref{f3}) \\
011011 & even&$2$&   & $flip(2)$  &(\ref{f2})\\
001011 & odd && $3$  & $flip(4)$  &(\ref{f5}) \\
001111 & even&$2$&   &$flip2(1, 2)$&(\ref{f3}) \\
111111 & even&$1$&   &$flip(1)$ &(\ref{f2})\\
011111& odd && $2$  &$flip(3)$ &(\ref{f5}) \\
010111 & even&$3$&  &$flip(2)$ & (\ref{f2})\\
000111& odd && $4$  &$flip(5)$ &(\ref{f5}) \\
000101 & even &$2$&   &$flip(2)$ &(\ref{f2})\\
010101& odd && $2$  & $flip2(2, 3)$&(\ref{f4}) \\
001101& odd && $3$   & $flip(4)$&(\ref{f5}) \\
001001& even&$3$&    &$flip(3)$ &(\ref{f2})\\
000001& odd&&  &$flip(6)$  &(\ref{f1})
\end{tabular}
\end{center}
%\vspace{-10}
\captionsetup{format=hang}
\caption{The necklaces of length 6 induced by successive applications the function $f$ starting from $000000$. %Observe that the resulting listing is the same as the BRGC order listing for necklaces of length 6.
The sixth column of the table lists out the corresponding rules in $f$ that apply to each necklace to obtain the next necklace. %The values $w(\alpha)$, $t$ and $\ell_\alpha$ correspond to the weight of $\alpha$, the leftmost possible position such that $flip_\alpha(t) \in {\bf S}$ and the position of the leftmost $1$ of $\alpha$ respectively.
}
\vspace{-2em}
\end{table}\label{table:brgc11}

Let $\FLIPTWO{\alpha}{i}{j}$ be the string obtained by complementing both $b_i$ and $b_j$.
When the context is clear we use $flip2(i, j)$ instead of $\FLIPTWO{\alpha}{i}{j}$.
Also, let $w(\alpha)$ denote the number of $1$s of $\alpha$.
We claim that the following function $f$ computes the next string in the cyclic ordering $\mathcal{BRGC}({\bf S})$:
{\footnotesize
%\begin{equation}
\begin{subnumcases}{\hspace{-2em}f(\alpha) =}
%flip(\alpha, 1)  & \mbox{if $w(\alpha)$ is even and $flip(\alpha, 1) \in {\bf S}$;} \\
0^n   & \mbox{if $\alpha = 0^{n-1}1$;}\label{f1}  \\ %and $0^n \in {\bf S}$;} \\
%0^{n-2}11   & \mbox{if $\alpha = 0^{n-1}1$ and $0^n \notin {\bf S}$;} \\
\FLIP{\alpha}{t_\alpha}   & \mbox{if $w(\alpha)$ is even and ($t_\alpha = 1$ or $\FLIPTWO{\alpha}{t_\alpha-1}{t_\alpha} \notin {\bf S}$);}\label{f2} \\
\FLIPTWO{\alpha}{t_\alpha-1}{t_\alpha}    & \mbox{if $w(\alpha)$ is even and $\FLIPTWO{\alpha}{t_\alpha-1}{t_\alpha} \in {\bf S}$;}\label{f3} \\
\FLIPTWO{\alpha}{\ell_\alpha}{\ell_\alpha +1}   &  \mbox{if $w(\alpha)$ is odd and $\FLIP{\alpha}{\ell_\alpha +1} \notin {\bf S}$;}\label{f4} \\
\FLIP{\alpha}{\ell_\alpha +1} & \mbox{if $w(\alpha)$ is odd and $\FLIP{\alpha}{\ell_\alpha +1} \in {\bf S}$.}\label{f5}  \label{eq:recursive}
\end{subnumcases}}
%\end{equation}

Thus, successive applications of the function $f$ on a flip-swap language ${\bf S}$, starting with the string $0^n$, list out each string in ${\bf S}$ in BRGC order.
As an illustration of the function $f$, successive applications of this rule for the set of necklaces of length $6$ starting with the necklace $000000$ produce the listing in Table 2.

\begin{restatable}{theorem}{mainclaim}
\label{thm:big}
If ${\bf S}$ is a flip-swap language including the string $0^n$ and $|{\bf S}| > 1$, then $f(\alpha)$ is the string immediately following the string $\alpha$ in ${\bf S}$ in the cyclic ordering $\mathcal{BRGC}({\bf S})$.
\end{restatable}

We will provide a detailed proof of this theorem in the next subsection.  Observe that each rule in $f$ complements at most two bits and thus successive strings in ${\bf S}$ differ
by at most two bit positions.  Observe that when $0^n$ is excluded from ${\bf S}$, then $\mathcal{BRGC}({\bf S})$ is still a $2$-Gray code (although not necessarily cyclic).
This proves Theorem~\ref{thm:main2}.

%====================================
%====================================
%====================================
%====================================
\subsection{Proof of Theorem~\ref{thm:big}}\label{sec:proof}% and Theorem~\ref{thm:2Gray}}\label{sec:proof}
%In this section, we prove
This section proves Theorem~\ref{thm:big}. %, Theorem~\ref{thm:2Gray} and Corollary~\ref{cor:2Gray}.
We begin with a lemma by Vajnovszki~\cite{vaj}, and a remark that is due to the fact that $0^{n-1}1$ is in a flip-swap language ${\bf S}$ when $|{\bf S}| > 1$.
%, Vajnovszki proved the following result. % related to BRGC order.
%======================
\begin{lemma}\label{lem:gray}
Let $\alpha = b_1 b_2 \cdots b_n$ and $\beta$ be length $n$ binary strings such that $\alpha \ne \beta$.  Let $r$ be the rightmost position in which $\alpha$ and $\beta$ differ.
Then $\alpha$ comes before $\beta$ in BRGC order (denoted by $\alpha \prec \beta$) if and only if $w(b_r b_{r+1} \cdots b_n)$ is even.
\end{lemma}
%======================

%Since the last string in ${\bf B}(n)$ in BRGC order is $0^{n-1}1$~\cite{ruskey}, the following remark follows.
%Since $0^{n-1}1$ is an element of a flip-swap language ${\bf S}$ when $|{\bf S}| > 1$, the following remark follows.
%Since a flip-swap language ${\bf S}$ is a subset of ${\bf B}(n)$ and $0^{n-1}1 \in {\bf S}$ when $|{\bf S}| > 1$, the following remark follows.
%Since $0^{n-1}1$ is also an element of ${\bf S} \in {\bf B}(n)$, the following remark follows.
\begin{remark}\label{lem:last}
%The first, second and last pseudo-necklaces and necklaces in BRGC order are $0^n$,  $0^{n-2}11$ and $0^{n-1}1$ respectively.
%The last string of a flip-swap language ${\bf S}$ in BRGC order is $0^{n-1}1$ when $|{\bf S}| > 1$.
A flip-swap language ${\bf S}$ in BRGC order ends with $0^{n-1}1$ when $|{\bf S}| > 1$.
\end{remark}

Let $succ({\bf S}, \alpha)$ be the \emph{successor} of $\alpha$ in ${\bf S}$ in BRGC order (i.e. the string after $\alpha$ in the cyclic ordering $\mathcal{BRGC}({\bf S})$).
%element in ${\bf S}$ that appears immediately after $\alpha$ in BRGC order.
%the element in ${\bf S}$ that appears immediately after $\alpha$ in BRGC order, that is the successor of $\alpha$ in ${\bf S}$ in BRGC order.
Next we provide two lemmas, and then prove Theorem~\ref{thm:big}.

\begin{lemma} \label{thm:even}
Let ${\bf S}$ be a flip-swap language with $|{\bf S}| > 1$ and $\alpha$ be a string in ${\bf S}$. % be the string immediately after $\alpha$ in BRGC order.
Let $t_\alpha$ be the leftmost position such that $\FLIP{\alpha}{t_\alpha} \in {\bf S}$.
If $w(\alpha)$ is even, then $t_\alpha$ is the rightmost position in which $\alpha$ and $succ({\bf S}, \alpha)$ differ.
\end{lemma}
\begin{proof}
By contradiction.
Let $\alpha = b_1 b_2 \cdots b_n$ and $\beta = succ({\bf S}, \alpha)$.
% = 0^j 1 a_{j+2} a_{j+3} \cdots a_n$ and
Let $r$ be the rightmost position in which $\alpha$ and $\beta$ differ with $r \neq t_\alpha$.
%There are two cases depending on the position of $t'$.
If $t_\alpha > r$, then $\beta$ has the suffix $1b_{r+1} b_{r+2} \cdots b_n$ since $b_r = 0$ because $r<t_\alpha \leq \ell_\alpha$.
Thus by the {flip-first} property, $0^{r-1}1b_{r+1} b_{r+2} = \FLIP{\alpha}{r} \in {\bf S}$ and $r<t_\alpha$, a contradiction.

Otherwise if $t_\alpha < r$, then let $\gamma = \FLIP{\alpha}{t_\alpha}$.
Clearly $\gamma \neq \alpha$.
Now observe that $w(b_t b_{t+1} \cdots b_n)$ is even because $t_\alpha \leq \ell_\alpha$ and $w(\alpha)$ is even, and thus by Lemma~\ref{lem:gray}, $\alpha \prec \gamma$.
%Thus, $\alpha \prec \beta \prec \gamma$ since $\beta$ is the successor of $\alpha$.
%Observe that $\alpha \prec \beta \prec \gamma$ since $\sum_{i=t}^n b_i$ is even and $\beta$ is the successor of $\alpha$.
%Thus, $\gamma$ has the suffix $b_{t'} b_{t'+1} \cdots b_n$.
Also, $\gamma$ has the suffix $b_{r} b_{r+1} \cdots b_n$ and $w(b_{r} b_{r+1} \cdots b_n)$ is even because $\alpha \prec \beta$ and $r$ is the rightmost position $\alpha$ and $\beta$ differ, and thus also by Lemma~\ref{lem:gray}, $\gamma \prec \beta$.
Thus $\alpha \prec \gamma \prec \beta$, a contradiction.
Therefore $r = t_\alpha$.
\end{proof}

\begin{lemma} \label{thm:odd}
Let ${\bf S}$ be a flip-swap language with $|{\bf S}| > 1$ and $\alpha\neq 0^{n-1}1$ be a string in ${\bf S}$. %, and $\beta = next({\bf S}, \alpha)$.
%Let $j$ be the largest value such that $0^j$ is a prefix of $\alpha$.
If $w(\alpha)$ is odd, then $\ell_\alpha +1$ is the rightmost position in which $\alpha$ and $succ({\bf S}, \alpha)$ differ.
\end{lemma}

%\vspace{-0.15in}
\begin{proof}
%The proof is similar to the proof of Lemma~\ref{thm:even}.
Since $\alpha \neq 0^{n-1}1$ and $w(\alpha)$ is odd, $\ell_\alpha < n-1$.
We now prove the lemma by contradiction.
Let $\alpha = b_1 b_2 \cdots b_n$ and $\beta =  succ({\bf S}, \alpha)$.
Let $r \neq \ell_\alpha + 1$ be the rightmost position in which $\alpha$ and $\beta$ differ.
%Observe that $t' > j+2$ because $w(a_{t'} a_{t'+1} \cdots a_n)$ is odd when $t' < j+1$, a contradiction to $\alpha \prec \beta$.
If $r < \ell_\alpha + 1$, then $w(b_{r} b_{r+1} \cdots b_n)$ is odd but $\alpha \prec \beta$, a contradiction by Lemma~\ref{lem:gray}.
%Let $\gamma = a_1 a_2 \cdots a_j \overline{a}_{j+1} \overline{a}_{j+2} a_{j+3} \cdots a_n$.
Otherwise if $r > \ell_\alpha + 1$, then let $\gamma = \FLIPTWO{\alpha}{\ell_\alpha}{\ell_\alpha + 1}$.
Clearly $\gamma \neq \alpha$, and by the {flip-first} and {swap-first} properties, $\gamma \in {\bf S}$.
Also, observe that $w(b_{\ell_\alpha + 1} b_{\ell_\alpha + 2} \cdots b_n)$ is even because $w(\alpha)$ is odd, and thus by Lemma~\ref{lem:gray}, $\alpha \prec \gamma$.
%Thus, $\alpha \prec \beta \prec \gamma$ since $\beta$ is the successor of $\alpha$.
%$\alpha \prec \beta \prec \gamma$ since $\sum_{i=j+2}^n b_i$ is even and $\gamma$ is the necklace just after $\beta$.
%Thus, $\gamma$ has the suffix $b_{t'} b_{t'+1} \cdots b_n$.
Further, $\gamma$ has the suffix $b_{r} b_{r+1} \cdots b_n$ and $w(b_{r} b_{r+1} \cdots b_n)$ is even because $\alpha \prec \beta$ and $r$ is the rightmost position $\alpha$ and $\beta$ differ, and thus also by Lemma~\ref{lem:gray}, $\gamma \prec \beta$.
Thus $\alpha \prec \gamma \prec \beta$, a contradiction.
Therefore $r = \ell_\alpha + 1$. % is the rightmost position in which $\alpha$ and $\beta$ differ.
%However, observe that $w(a_{t'} a_{t'+1} \cdots a_n)$ is odd since $\alpha$ and $\beta$ differ at position $t'$, a contradiction to $\beta \prec \gamma$.
\end{proof}

%Now we use Lemma~\ref{thm:even} and Lemma~\ref{thm:odd} to prove Theorem~\ref{thm:big}.

\begin{proof}[Proof of Theorem~\ref{thm:big}]
Let $\alpha = a_1 a_2 \cdots a_n$ and $\beta = succ({\bf S}, \alpha) = b_1 b_2 \cdots b_n$.
Let $t_\alpha$ be the leftmost position such that $\FLIP{\alpha}{t_\alpha} \in {\bf S}$.
First we consider the case when $\alpha = 0^{n-1}1$.
Recall that the first string in ${\bf B}(n)$ in BRGC order is $0^n$~\cite{ruskey} and $0^n$ is a string in ${\bf S}$ by Lemma~\ref{lem:prefix}.
Also, the last string in ${\bf S}$ in BRGC order is $0^{n-1}1$ by Remark~\ref{lem:last} when $|{\bf S}| > 1$.
Thus the string that appears immediately after $\alpha$ in the cyclic ordering $\mathcal{BRGC}({\bf S})$ is $f(\alpha)$ when $\alpha = 0^{n-1}1$.
In the remainder of the proof, $\alpha \neq 0^{n-1}1$ and we consider the following two cases.
\begin{description}
\item[Case 1:] $w(\alpha)$ is even:\
If $t_\alpha=1$, then clearly $\beta = \FLIP{\alpha}{t_\alpha}= f(\alpha)$.
For the remainder of the proof, $t_\alpha>1$.

Since $t_\alpha \leq \ell_\alpha$, $\FLIPTWO{\alpha}{t_\alpha-1}{t_\alpha}$ has the prefix $0^{t_\alpha-2}1$.
We now consider the following two cases.
If $\FLIPTWO{\alpha}{t_\alpha-1}{t_\alpha} \notin {\bf S}$, then $\FLIP{\alpha}{t_\alpha}$ is the only string in ${\bf S}$ that has $t_\alpha$ as the rightmost position that differ with $\alpha$ and has the prefix $0^{t-2}$.
Therefore, $\beta = \FLIP{\alpha}{t_\alpha} = f(\alpha)$.
Otherwise, $\FLIPTWO{\alpha}{t_\alpha-1}{t_\alpha}$ and $\FLIP{\alpha}{t_\alpha}$ are the only strings in ${\bf S}$ that have $t_\alpha$ as the rightmost position that differ with $\alpha$ and have the prefix $0^{t_\alpha-2}$.
By Lemma~\ref{lem:gray}, $\FLIPTWO{\alpha}{t_\alpha-1}{t_\alpha} \prec \FLIP{\alpha}{t_\alpha}$ since $w(1\overline{a}_{t_\alpha}  a_{t_\alpha+1} a_{t_\alpha+2} \cdots a_n)$ is even. %$w(0^{t-2} 1 \overline{b_t} b_{t+1})$ is even.
Thus, $\beta = \FLIPTWO{\alpha}{t_\alpha-1}{t_\alpha} = f(\alpha)$.
\item[Case 2:] $w(\alpha)$ is odd:\ % and $\alpha \neq 0^{n-1}1$:\
By Lemma~\ref{thm:odd}, $\beta$ has the suffix $\overline{a}_{\ell_\alpha + 1} a_{\ell_\alpha + 2} a_{\ell_\alpha + 3} \cdots a_n$.
If $\FLIP{\alpha}{\ell_\alpha+1} \notin {\bf S}$, then by the {flip-first} and {swap-first} properties, $\FLIPTWO{\alpha}{\ell_\alpha}{\ell_\alpha + 1}$ is the only string in ${\bf S}$ that has $\ell_\alpha + 1$ as the rightmost position that differ with $\beta$.
Thus, $\beta = \FLIPTWO{\alpha}{\ell_\alpha}{\ell_\alpha + 1} = f(\alpha)$.
Otherwise by Lemma~\ref{lem:gray}, any string $\gamma \in {\bf S}$ with the suffix $\overline{a}_{\ell_\alpha + 1} a_{\ell_\alpha + 2} a_{\ell_\alpha + 3} \cdots a_n$
and $\gamma \neq \FLIP{\alpha}{\ell_\alpha+1}$ has $\FLIP{\alpha}{\ell_\alpha +1} \prec \gamma$ because $w(1 \overline{a}_{ \ell_\alpha + 1} a_{\ell_\alpha + 2}a_{\ell_\alpha + 3} \cdots a_n)$ is even.
Thus, $\beta = \FLIP{\alpha} {\ell_\alpha + 1} = f(\alpha)$.
\end{description}
Therefore, the string immediately after $\alpha$ in the cyclic ordering $\mathcal{BRGC}({\bf S})$ is $f(\alpha)$.
\end{proof}

\section{Generation algorithm for flip-swap languages} \label{sec:gen_n}
In this section we present a generic algorithm to generate  $2$-Gray codes for flip-swap languages via the function $f$.
%The algorithm can also generate $2$-Gray codes for near flip-swap languages.  %over ${\bf B}(n)$.

A na\"{i}ve approach to implement $f$ is to find $t_\alpha$ by test flipping each bit in $\alpha$ to see if the result is also in the set when $w(\alpha)$ is even; or test flipping the ($\ell_\alpha+1$)-th bit of $\alpha$ to see if the result is also in the set when $w(\alpha)$ is odd.
Since $t_\alpha \leq \ell_\alpha$, we only need to examine the length $\ell_\alpha-1$ prefix of $\alpha$ to find $t_\alpha$.
Such a test can be done in $O(nm)$ time, where $O(m)$ is the time required to complete the membership test of the set under consideration.
Pseudocode of the function $f$ is given in Algorithm~\ref{alg:n2-f}.

To list out each string of a flip-swap language ${\bf S}$ in BRGC order, we can  repeatedly apply the function $f$ until it reaches the starting string.
We also maintain $w(\alpha)$ and $\ell_\alpha$ which can be easily maintained in $O(n)$ time for each string generated.
%To make the algorithm capable to generate near flip-swap languages over ${\bf B}(n)$,
We also add a condition to avoid printing the string $0^n$ if $0^n$ is not a string in ${\bf S}$.
Pseudocode for this algorithm, starting with the string $0^n$, is given in Algorithm~\ref{alg:n2-decode}.
The algorithm can easily be modified to generate the corresponding counterpart of ${\bf S}$ with respect to $0$.

%A complete C implementation of the algorithm is given in the Appendix.

A simple analysis shows that the algorithm generates ${\bf S}$ in $O(nm)$-time per string.
A more thorough analysis improves this to $O(n+m)$-amortized time per string.

\begin{theorem} \label{lem:n2}
%The algorithm \textit{BRGC} lists out each string of a flip-swap language or a near flip-swap language ${\bf S}$ over ${\bf B}(n)$ in BRGC order in $O(n+m)$-amortized time per string, where $O(m)$ is the time required to complete the membership tester for ${\bf S}$.
%function $f$ can be computed in $O(n^2)$ time for %sets discussed in Section~\ref{sec:examples}.
If ${\bf S}$ is a flip-swap language, then the algorithm  \textit{BRGC}  produces
$\mathcal{BRGC}(S)$ in $O(n+m)$-amortized time per string, where $O(m)$ is the time required to complete the membership tester for ${\bf S}$.

\end{theorem}

\begin{proof}
%It is easy to see that the main overhead of the algorithm \textit{BRGC} is the function $f$.
Let $\alpha = a_1 a_2 \cdots a_n$ be a string in  ${\bf S}$.
Clearly $f$ can be computed in $O(n)$ time when $w(\alpha)$ is odd.
Otherwise when $w(\alpha)$ is even, the {\bf while} loop in line 5 of Algorithm~\ref{alg:n2-f} performs a membership tester on each string $\beta = b_1 b_2 \cdots b_n$ in ${\bf S}$ with $b_{\ell_\alpha } b_{\ell_\alpha + 1} \cdots b_n = a_{\ell_\alpha} a_{\ell_\alpha + 1} \cdots a_n$ and $w(b_1 b_2 \cdots b_{\ell_\alpha -1}) = 1$.
Observe that each of these strings can only be examined by the membership tester once, or otherwise the {\bf while} loop in line 5 of Algorithm~\ref{alg:n2-f} produces the same $t_\alpha$ which results in a duplicated string, a contradiction.
Thus, the total number of membership testers performed by the algorithm is bound by $|{\bf S}|$, and therefore $f$ runs in $O(m)$-amortized time per string.
Finally, since the other part of the algorithm runs in $O(n)$ time per string, the algorithm \textit{BRGC} runs in $O(n+m)$-amortized time per string.
\end{proof}

The membership tests in this paper can be implemented in $O(n)$ time and $O(n)$ space; see~\cite{Booth,journals/jal/Duval83,Sawada201346} for necklaces, Lyndon words, prenecklaces and pseudo-necklaces of length $n$.
One exception is the test for prefix normal words of length $n$, which requires $O(n^{1.864})$ time and $O(n)$ space~\cite{Chan:2015:CIV:2746539.2746568}.
Together with the above theorem, this proves Theorem~\ref{thm:main3}.

%A C program that generates all of our listed flip-swap languages is available on
Visit the Combinatorial Object Server~\cite{cos} for a C implmentation of our algorithms.

\begin{algorithm}[t]
\footnotesize
\caption{Pseudocode of the implementation of the function $f$.}
\label{alg:n2-f}

\vspace{-0.6em}
\begin{algorithmic} [1]
\Statex

\Function{$f$}{$\alpha$}

	\If{$\alpha = 0^{n-1}1$} \ $\FLIP{\alpha}{n}$
	\ElsIf{$w(\alpha)$ is even}

		\State $t_\alpha \gets \ell_\alpha$
		\While{$t_\alpha>1$  {\bf and} $\FLIP{\alpha}{t_\alpha-1} \in {\bf S}$} \ $t_\alpha \gets t_\alpha - 1$
		\EndWhile

				\If{$t_\alpha \neq 1$  {\bf and} $\FLIPTWO{\alpha}{t_\alpha-1}{t_\alpha} \in {\bf S}$} \ $\alpha \gets \FLIPTWO{\alpha}{t_\alpha-1}{t_\alpha}$
				\Else \ $\alpha \gets \FLIP{\alpha}{t_\alpha}$
				\EndIf

	\Else

		\If{$\FLIP{\alpha}{\ell_\alpha+1} \notin {\bf S}$} \ $\alpha \gets \FLIPTWO{\alpha}{\ell_\alpha}{\ell_\alpha+1}$
		\Else \ $\alpha \gets \FLIP{\alpha}{\ell_\alpha+1}$
		\EndIf
	\EndIf

\EndFunction
\end{algorithmic}
\end{algorithm}

\begin{algorithm}[t]
\footnotesize
\caption{Algorithm to list out each string of a flip-swap language ${\bf S}$ in BRGC order.}
\label{alg:n2-decode}

\vspace{-0.6em}
\begin{algorithmic} [1]
\Statex

\Procedure{\textit{BRGC}}{}
%\State $\textrm{Print($0$)}$
%\State{$\alpha = a_1 a_2 \cdots a_n \gets 0^n$}
\State{$\alpha = b_1 b_2 \cdots b_n \gets 0^n$}
%\While{$\alpha \neq 0^{n-1}1$}
\Do
%	\EndIf
	\If{$\alpha \neq 0^n$ {\bf or} $0^n \in {\bf S}$} \ $\textrm{Print$(\alpha)$}$ \EndIf
	\State $f(\alpha)$
	\State $w(\alpha) \gets 0$
	\For{$i$ {\bf from} $n$ {\bf down to} $1$}
			\If{$b_i = 1$}  \ $w(\alpha) \gets w(\alpha) + 1$ \EndIf
			\If{$b_i = 1$}  \ $\ell_\alpha \gets i$ \EndIf
	\EndFor

\doWhile{$\alpha \neq 0^{n}$}
%\EndWhile
%\State $\textrm{Print$(0^{n-1}1)$}$
\EndProcedure
\end{algorithmic}
\end{algorithm}

{
\footnotesize
\bibliographystyle{abbrv}
\bibliography{refs}}
%==================================================

\newpage

\noindent

{\large \bf Appendix: Proofs for flip-swap languages}
\\

This section provides the proofs for the remainder of the languages in Theorem~\ref{thm:main1}.
%In this section we provides the proofs for flip-swap languages and near flip-swap languages.
%Sections~\ref{subsec:binary-strings}-\ref{subsec:union-intersection} \
%By Theorem~\ref{thm:big}, the function $f$ generates the next string of a the flip-swap language in BRGC order, and the listing of elements of a flip-swap language in BRGC order induces a cyclic $2$-Gray code.
%For near flip-swap languages, the function $f$ generates the next string of ${\bf S} \cup \{0^n\}$ of a near flip-swap languages ${\bf S}$ in BRGC order, and the listing of elements of a near flip-swap languages in BRGC order induces a $2$-Gray code.
%For related concepts such as strings of length $n$ with forbidden $01^t$ or $10^t$, we will focus on the one with respect to $1$.  \red{I do not understand this sentence - can we remove?  Also, may be good to remind reader, that unless otherwise stated, the discussion of flip-swap languages are with respect to 1}
%Denote the set of binary strings of length $n$ by ${\bf B}(n)$.
Unless otherwise stated, the discussion of flip-swap languages are with respect to 1.
%This section shows that each of the combinatorial objects is a flip-swap language.

\subsection*{Binary strings with weight $\leq k$}
Recall the weight of a binary string is the number of 1s it contains.
%The weight of a binary string $\alpha$ is the number of $1$s in $\alpha$.
Let ${\bf S}$ be the set of binary strings of length $n$  having weight less than or equal to some $k$.
%As discussed in Section~\ref{subsec:binary-strings}, applying the flip-first-1 or swap-first-1 operation on a binary string yield another binary string in ${\bf B}(n)$.
%Furthermore, both flip-first-1 and swap-first-1 operations
Observe that ${\bf S}$  satisfies the two closure properties of a flip-swap language as the flip-first and swap-first operations either decrease or maintain the weight.
Thus, ${\bf S}$ is a flip-swap language.
%Observe that flipping the leftmost $1$ of $\alpha$ or swapping the leftmost $1$ of $\alpha \neq 0^{n-1}1$ with the bit on the right does not increase the weight of the resulting string. % and yields a string in ${\bf S}$. %, which flipping the leftmost $1$ actually decrements the weight of $\alpha$ by one.
%Thus ${\bf S}$ satisfies the flip-first and the swap-first properties and is a flip-swap language over ${\bf B}(n)$.
%Therefore, we can list the elements in ${\bf S}$ so that each successive strings differ by at most two bits cyclicly.

%Similarly, the set of binary strings of length $n$ with each string  having weight that is greater than or equal to $k$  is a flip-swap language  with respect to $0$.
%Thus, we can similarly list the elements in the set so that each successive strings differ by at most two bits cyclicly.

\subsection*{Binary strings $\leq \gamma$}
Let ${\bf S}$ be the set of binary strings of length $n$ with each string lexicographically smaller or equal to some string $\gamma$.
Observe that ${\bf S}$ satisfies the two closure properties of a flip-swap language as the flip-first and swap-first operations either make the resulting string lexicographically smaller or produce the same string.
Thus, ${\bf S}$ is a flip-swap language.
%As discussed in Section~\ref{subsec:binary-strings}, applying the flip-first-1 or swap-first-1 operation on a binary string yield another binary string in ${\bf B}(n)$.
%Furthermore, both flip-first-1 and swap-first-1 operations
%Consider a string $\alpha \neq 0^n$ in ${\bf S}$. Observe that flipping the leftmost $1$ of $\alpha$ makes the resulting string lexicographically smaller, while swapping the leftmost $1$ of $\alpha \neq 0^{n-1}1$ with the bit on the right either produces the same string, or a string that is lexicographically smaller.
%Thus the resulting string is lexicographically smaller than $\gamma$ and in ${\bf S}$.
%Thus ${\bf S}$ satisfies the flip-first and the swap-first properties and is a flip-swap language over ${\bf B}(n)$.
%Therefore, we can list the elements in ${\bf S}$ so that each successive strings differ by at most two bits cyclicly. \red{Do we need this statement for all languages? Really we are just proving they are flip-swap languages, so it is a bit redundant}

%Similarly, the set of binary strings of length $n$ with each string lexicographically larger or equal to some string $\gamma$ is a flip-swap language with respect to $0$.
%Thus, we can similarly list the elements in the set so that each successive strings differ by at most two bits cyclicly.

%\subsection{Binary strings with forbidden $10^k$}

\subsection*{Binary strings with $\leq k$ inversions}
Recall that an \emph{inversion} with respect to $0^*1^*$ in a binary string $\alpha = b_1 b_2 \cdots b_n$ is any $b_i = 1$ and $b_j = 0$ such that $i < j$.
For example when $\alpha = 100101$, it has $4$ inversions: $(b_1, b_2), (b_1, b_3), (b_1, b_5), (b_4, b_5)$.
Let ${\bf S}$ be the set of binary strings of length $n$ with less than or equal to $k$ inversions with respect to $0^*1^*$. % and $\alpha =  b_1 b_2 \cdots b_n \neq 0^n$ be a string in ${\bf S}$.
Observe that ${\bf S}$ satisfies the two closure properties of a flip-swap language as the flip-first and swap-first operations either decrease or maintain the number of inversions.
Thus, ${\bf S}$ is a flip-swap language.
%Observe that flipping the leftmost $1$ of $\alpha$ reduces the number of inversions by the number of $0$s in $b_{\ell_\alpha +1} b_{\ell_\alpha+2} \cdots b_n$. %the inversions that involves $a_i$.
%Similarly, swapping the leftmost $1$ of $\alpha \neq 0^{n-1}1$ with the bit on the right either produces the same string, or reduces the number of inversions by one.
%$a_j$ such that $\ell(\alpha) < j$ and $a_j = 0$ reduces the number of inversions by the number of $0$s in $a_{\ell(\alpha)  +1} a_{\ell(\alpha) +2} \cdots a_j$. %$j-\alpha_1$.
%Thus ${\bf S}$ satisfies the flip-first and the swap-first properties and is a flip-swap language over ${\bf B}(n)$.
%Therefore, we can list the elements in ${\bf S}$ so that each successive strings differ by at most two bits cyclicly.

%An {inversion} with respect to $1^*0^*$ in a binary string $\alpha = b_1 b_2 \cdots b_n$ is any $b_i = 0$ and $b_j = 1$ such that $i < j$.
%Similarly, the set of binary strings of length $n$ with less than or equal to $k$ inversions with respect to $1^*0^*$ is a flip-swap language with respect to $0$.
%Thus, we can similarly list the elements in the set so that each successive strings differ by at most two bits cyclicly.

\subsection*{Binary strings with $\leq k$ transpositions}
%The number of inversions of a string can be deemed as the number of ``adjacent-transposition" required to sort the string into the form of $0^*1^*$.
%If we remove the ``adjacent" criteria, then we can consider the number of transpositions required to sort a string into the form $0^*1^*$.
Recall that the number of \emph{transpositions} of a binary string $\alpha = b_1 b_1 \cdots b_n$ with respect to $0^*1^*$ is the minimum number of $swap(i, j)$ operations required to change $\alpha$ into the form $0^*1^*$.
%Hamming distance between $\alpha$ and the string $0^{n-w(\alpha)} 1^{w(\alpha)}$.
For example, the number of transpositions of the string $100101$ is $1$. %requires $4$ ``adjacent-transpositions" or inversions to sort into the form of $0^*1^*$, but it only requires $1$ transpositions to sort it: namely swapping the leftmost $1$ with the $0$ in position $5$.
Let ${\bf S}$ be the set of binary strings of length $n$ with less than or equal to $k$ transpositions with respect to $0^*1^*$. % and $\alpha \neq 0^n$ be a string in ${\bf S}$.
%Observe that flipping the leftmost $1$ of $\alpha$ reduces the number of transpositions of the resulting string by one if $\alpha \neq 0^{n-w(\alpha)} 1^{w(\alpha)}$, while the number of transpositions remains zero if $\alpha$ is of the form $0^*1^*$.
%Similarly, swapping the leftmost $1$ of $\alpha \neq 0^{n-1}1$ with the bit on the right either keeps the number of transpositions of the resulting string the same, or reduces the number of transpositions by one.
%Thus ${\bf S}$ satisfies the flip-first and the swap-first properties and is a flip-swap language over ${\bf B}(n)$.
%Therefore, we can list the elements in ${\bf S}$ so that each successive strings differ by at most two bits cyclicly.
Observe that ${\bf S}$ satisfies the two closure properties of a flip-swap language as the flip-first and swap-first operations either decrease or maintain the number of transpositions.
Thus, ${\bf S}$ is a flip-swap language.

%The number of \emph{transpositions} of a binary string $\alpha = b_1 b_1 \cdots b_n$ with respect to $1^*0^*$ is the minimum number of $swap(i, j)$ operations required to change $\alpha$ into the form $1^*0^*$.
%Similarly, the set of binary strings of length $n$ with less than or equal to $k$ transpositions with respect to $1^*0^*$ is a flip-swap language with respect to $0$.
%Thus, we can similarly list the elements in the set so that each successive strings differ by at most two bits cyclicly.

\subsection*{Binary strings $< \text{or} \leq$ their reversal}
Let ${\bf S}$ be the set of binary strings of length $n$ with each string lexicographically smaller than their reversal. %and $\alpha \neq 0^n$ be a string in ${\bf S}$.
%Consider a string $\alpha \neq 0^n$ in ${\bf S}$.
%Notice that ${\bf S}$ is not a flip-swap language since $0^{n-1}1$ is a string in ${\bf S}$ but $0^n$ is not.
%We now show that ${\bf S}$ is a near flip-swap language.
Observe that ${\bf S}$ satisfies the swap-first property as the swap-first operation either produces the same string, or makes the resulting sting lexicographically smaller while its reversal lexicographically larger.
%Observe that swapping the leftmost $1$ of $\alpha \neq 0^{n-1}1$ with the bit on the right either produces the same string, or makes the resulting sting lexicographically smaller while its reversal lexicographically larger. Thus ${\bf S}$ satisfies the swap-first property.
Furthermore, ${\bf S} \cup \{0^n\}$ satisfies the flip-first property as the flip-first operation complements the most significant bit of $\alpha$ but the least significant bit of its reversal when $w(\alpha) > 1$; or otherwise produces the string $0^n$ when $w(\alpha) = 1$.
Thus, ${\bf S}$ is a flip-swap language. % and ${\bf S}$ is a near flip-swap language.
%Since $\alpha$ is lexicographically smaller than (or equal to) its reversal, the resulting string must also be less than its reversal when $w(\alpha) > 1$; or otherwise produces the string $0^n$ when $w(\alpha) = 1$.
%Thus ${\bf S} \cup \{0^n\}$ also satisfies the flip-first property and ${\bf S}$ is a near flip-swap language over ${\bf B}(n)$.
%Such
%Therefore, we can list the elements in ${\bf S}$ so that each successive strings differ by at most two bits.
%On the contrary, the set of binary strings of length $n$ with each string lexicographically smaller than or equal to their reversal is a flip-swap language since now $0^n$ is in the set. 
The proof for the set of binary strings of length $n$ with each string lexicographically smaller than or equal to their reversal is similar to the proof for ${\bf S}$.

Equivalence class of strings under reversal has also been called neckties~\cite{Savage:1997:SCG:273590.273592}.

%Similarly, the set of binary strings of length $n$ with each string lexicographically larger than their reversal is a near flip-swap language with respect to $0$, while the set of binary strings of length $n$ with each string lexicographically larger than or equal to their reversal is a flip-swap language with respect to $0$.
%Thus, we can similarly list the elements in the set so that each successive strings differ by at most two bits.
%Therefore binary strings with less than or equal to $k$ transpositions with respect to $1^*0^*$ form a first-$0$ $2$-BRGC language.
%Similarly, the set of binary strings with each string lexicographically smaller than their reversal is also closed under the swap-first-$1$ operation.
%Further, flipping the leftmost $1$ of $\alpha$ produces

%\subsection{Binary strings without the prefix $1\beta$}

\subsection*{Binary strings $< \text{or} \leq$ their complemented reversal}\label{subsec:complement}
Let ${\bf S}$ be the set of binary strings of length $n$ with each string lexicographically smaller than (or equal to) its complemented reversal.  %and $\alpha \neq 0^n$ be a string in ${\bf S}$.
%Consider a string $\alpha \neq 0^n$ in ${\bf S}$.
%Observe that flipping the leftmost $1$ of $\alpha$ makes the resulting string lexicographically smaller while its complemented reversal lexicographically larger when $w(\alpha) > 1$; or otherwise produces the string $0^n$ when $w(\alpha) = 1$.
Observe that ${\bf S} $ satisfies the flip-first property as the  flip-first operation makes the resulting string lexicographically smaller while its complemented reversal lexicographically larger.
%Thus  ${\bf S} \cup \{0^n\}$ satisfies the flip-first property.
Furthermore, ${\bf S} $ satisfies the swap-first property as the swap-first operation either produces the same string, or complements the most significant bit of $\alpha$ and also a $1$ of its complemented reversal.
Thus, the resulting string must also be less than its complemented reversal.
Thus, ${\bf S}$ is a flip-swap language.
%Thus  ${\bf S} \cup \{0^n\}$ also satisfies the swap-first property and  ${\bf S}$ is a near flip-swap language over ${\bf B}(n)$.
%Therefore, we can list the elements in ${\bf S}$ so that each successive strings differ by at most two bits.

%Similarly, the set of binary strings of length $n$ with each string lexicographically larger than (or equal to) its complemented reversal is a flip-swap language with respect to $0$.
%Thus, we can similarly list the elements in the set so that each successive strings differ by at most two bits.

\subsection*{Binary strings with forbidden $10^t$}\label{subsec:forbidden}
Let ${\bf S}$ be the set of binary strings of length $n$ without the substring $10^t$. %and $\alpha \neq 0^n$ be a string in ${\bf S}$.
%Consider a string $\alpha \neq 0^n$ in ${\bf S}$.
Observe that ${\bf S}$ satisfies the two closure properties of a flip-swap language as the flip-first and swap-first operations do not create the substring $10^t$.
Thus, ${\bf S}$ is a flip-swap language.
%Observe that flipping the leftmost $1$ of $\alpha$ does not create the substring $10^t$ and thus ${\bf S}$ satisfies the flip-first property.
%Similarly, swapping the leftmost $1$ of $\alpha \neq 0^{n-1}1$ with the bit on the right either produces the same string, or decrements the $0$s of the first substring of the format $10^*$.
%Thus ${\bf S}$ also satisfies the swap-first property and is a flip-swap language over ${\bf B}(n)$.
%Therefore, we can list the elements in ${\bf S}$ so that each successive strings differ by at most two bits cyclicly.

%Similarly, the set of binary strings of length $n$ without the substring $01^t$ is a flip-swap language with respect to $0$.
%Thus, we can similarly list the elements in the set so that each successive strings differ by at most two bits cyclicly.

\subsection*{Binary strings with forbidden prefix $1\gamma$}\label{subsec:forbidden-prefix}
Let ${\bf S}$ be the set of binary strings of length $n$ without the prefix $1\gamma$. % and $\alpha \neq 0^n$ be a string in ${\bf S}$.
%Consider a string $\alpha \neq 0^n$ in ${\bf S}$.
Observe that ${\bf S}$ satisfies the two closure properties of a flip-swap language as the flip-first and swap-first operations either create a string with the prefix $0$ or produce the same string.
Thus, ${\bf S}$ is a flip-swap language.
%Observe that flipping the leftmost $1$ of $\alpha$ creates a string that starts with the prefix $0$, that is a string without the prefix $1\gamma$ and thus ${\bf S}$ satisfies the flip-first property.
%Similarly, swapping the leftmost $1$ of $\alpha \neq 0^{n-1}1$ with the bit on the right either produces the same string, or produces a string that starts with the prefix $0$.
%Thus ${\bf S}$ also satisfies the swap-first property and is a flip-swap language over ${\bf B}(n)$.
%Therefore, we can list the elements in ${\bf S}$ so that each successive strings differ by at most two bits cyclicly.

%Similarly, the set that of binary strings of length $n$ without the prefix $0\gamma$ is a flip-swap language with respect to $0$.
%Thus, we can similarly list the elements in the set so that each successive strings differ by at most two bits cyclicly.

%The set ${\bf S}$ is closed under the flip-first-1 operation since

\subsection*{Lyndon words}
%A string $\alpha$ is \emph{periodic} if $\alpha = \beta^t$ for some string $\beta$ and $t \geq 2$.
%If a string is not periodic, it is said to be \emph{aperiodic}.
%A \emph{necklace} is the lexicographically smallest string in equivalence class of strings under rotation.
%A \emph{Lyndon word} is an aperiodic necklace.
%Let ${\bf S}$ be the set of Lyndon words of length $n$.
Let ${\bf L}(n)$ denote the set of Lyndon words of length $n$.
%Notice that ${\bf L}(n)$ is not a flip-swap language since $0^{n-1}1$ is a Lyndon word but $0^n$ is not.
%We now show that ${\bf L}(n)$ is a near flip-swap language over ${\bf B}(n)$.
%In Section~\ref{subsec:necklaces} we prove that applying the flip-first or the swap-first operation on a necklace yields a necklace.
Since ${\bf N}(n)$ is a flip-swap language and ${\bf L}(n) \cup \{0^n\} \subseteq {\bf N}(n)$, %is a subset of the set of necklaces,
it suffices to show that applying the flip-first or the swap-first operation on a Lyndon word either yields an aperiodic string or the string $0^n$. % that is aperiodic.
%the necklaces produced by applying the operations are aperiodic.

Clearly ${\bf L}(n) \cup \{0^n\}$ satisfies the two closure properties of a flip-swap language when $\alpha \in \{0^n, 0^{n-1}1\}$. Thus in the remaining of the proof, $\alpha \notin \{0^n, 0^{n-1}1\}$.
We first prove by contradiction that ${\bf L}(n) \cup \{0^n\}$ satisfies the flip-first closure property.
Let $\alpha = 0^j 1 b_{j+2} b_{j+3}\cdots b_n$ be a string in ${\bf L}(n) \cup \{0^n\}$.
%It is easy to see that $flip(\alpha, s(\alpha)) \in {\bf L} \cup \{0^n\}$ when $\alpha \in \{0^n, 0^{n-1}1\}$.
%Thus in the remaining of the proof we assume $\alpha \notin \{0^n, 0^{n-1}1\}$.
%If $\alpha = 0^{n-1}1$, then $flip(\alpha, s(\alpha)) = 0^n \in {\bf L} \cup \{0^n\}$ and thus ${\bf L} \cup \{0^n\}$ is closed under the flip-first-$1$ operation when $\alpha = 0^{n-1}1$.
Suppose that ${\bf L}(n) \cup \{0^n\}$ does not satisfy the flip-first closure property and $\FLIP{\alpha}{\ell_\alpha}$ is periodic.
Thus $\FLIP{\alpha}{\ell_\alpha} = (0^{j+1} \beta)^t$ for some string $\beta$ and $t\geq 2$.
Observe that $\alpha = 0^{j}1 \beta (0^{j+1} \beta)^{t-1}$ which is clearly not a Lyndon word, a contradiction.
Therefore ${\bf L}(n) \cup \{0^n\}$ satisfies the flip-first closure property.

Then similarly we prove by contradiction that ${\bf L}(n) \cup \{0^n\}$ satisfies the swap-first property.
If $b_{j+2} = 1$, then applying the swap-first operation on $\alpha$ produces the same Lyndon word.
Thus in the remaining of the proof, $b_{j+2} = 0$.
Suppose that ${\bf L}(n) \cup \{0^n\}$ does not satisfy the swap-first closure property such that $\alpha \in {\bf L}(n) \cup \{0^n\}$ but $\SWAP{\alpha}{\ell_\alpha}{\ell_\alpha+1}$ is periodic.
Thus $\SWAP{\alpha}{\ell_\alpha}{\ell_\alpha +1} = (0^{j+1}1 \beta)^t$ for some string $\beta$ and $t \geq 2$.
Thus $\alpha$ contains the prefix $0^{j}1$ but also the substring $0^{j+1}1$ in its suffix which is clearly not a Lyndon word, a contradiction.
%$\alpha = 0^{j}1 \gamma (0^{j+1}1 \beta)^{t-1}$ for some string $\gamma$ %and has the prefix $0^{j} 1$ but also contains the substring $0^{j+1}$,
%Thus, ${\bf L}(n) \cup \{0^n\}$ is a flip-swap language and ${\bf L}(n)$  is a near flip-swap language.
Thus, ${\bf L}(n)$ is a flip-swap language.
%Therefore, we can list Lyndon words of length $n$ so that each successive Lyndon words differ by at most two bits.

%Similarly, the set that contains each of the largest string in equivalence class of aperiodic strings of length $n$ under rotation is a near flip-swap language  with respect to $0$. %Thus, we can similarly list the elements in the set so that each successive strings differ by at most two bits.

In~\cite{vaj}, Vajnovszki proved that the BRGC order induces a cyclic $2$-Gray code for the set of Lyndon words of length $n$.

\subsection*{Prenecklaces}

Recall that a string $\alpha$ is a \emph{prenecklace} if it is a prefix of some necklace.
%Let ${\bf S}$ be the set of length $n$ prenecklaces.
In Section~\ref{sec:flip-swap} we prove that applying the flip-first or the swap-first operation on a necklace yields a necklace.
Thus by the definition of prenecklace, applying the flip-first or the swap-first operation on a prenecklace also creates a string that is a prefix of a necklace.
Thus, the set of prenecklaces of length $n$ is a flip-swap language.
%Thus the set of prenecklaces of length $n$ is closed under the flip-first and the swap-first operations and is a flip-swap language over ${\bf B}(n)$. %member of ${\bf L}$. %${\bf S} \in {\bf L}$.
%Therefore, we can list prenecklaces of length $n$ so that each successive prenecklaces differ by at most two bits cyclicly.

%Similarly, the set that contains length $n$ prefixes of each of the largest string in equivalence class of binary strings under rotation is a flip-swap language with respect to $0$.
%Thus, we can similarly list the elements in the set so that each successive strings differ by at most two bits cyclicly.

\subsection*{Pseudo-necklaces}\label{subsec:pseudonecklaces}
Recall that a \emph{block} with respect to $0^*1^*$ is a maximal substring of the form $0^*1^*$.
Each block $B_i$ with respect to $0^*1^*$ can be represented by two integers $(s_i, t_i)$ corresponding to the number of $0$s and $1$s respectively. For example, the string $\alpha = 000110100011001$ can be represented by $B_4B_3B_2B_1 = (3, 2)(1, 1)(3, 2)(2, 1)$.
A block $B_i = (s_i, t_i)$ is said to be \emph{lexicographically smaller} than a block $B_j = (s_j, t_j)$ (denoted by $B_i < B_j$) if $s_i < s_j$ or $s_i = s_j$ with $t_i < t_j$.

A string $\alpha = b_1b_2 \cdots b_n = B_b B_{b-1} \cdots B_1$ is a \emph{pseudo-necklace} with respect to $0^*1^*$ if $B_b \leq B_i$ for all $1 \leq  i < b$.
Observe that the set of pseudo-necklaces of length $n$ satisfies the two closure properties of a flip-swap language as the flip-first and swap-first operations do not make the first block $B_b$ lexicographically larger, while the remaining blocks either remain the same or become lexicographically larger.
Thus, the set of pseudo-necklaces of length $n$ is a flip-swap language.
%Thus the set of pseudo-necklaces of length $n$ satisfies the first-flip and the swap-first properties and is a flip-swap language over ${\bf B}(n)$.
%Therefore, we can list pseudo-necklaces of length $n$ so that each successive pseudo-necklaces differ by at most two bits cyclicly.
%Pseudo-necklaces were first defined in [11] and they are used as a stepping stone in our algorithms for necklaces and Lyndon words.

%Similarly, the set that contains
%A {block} with respect to $1^*0^*$ is a maximal substring of the form $1^*0^*$.
%Similarly, each block $B_i$ with respect to $1^*0^*$ can be represented by two integers $(t_i, s_i)$ corresponding to the number of $1$s and $0$s respectively.
%A string $\alpha = b_1b_2 \cdots b_n = B_b B_{b-1} \cdots B_1$ is a {pseudo-necklace} with respect to $1^*0^*$ if $B_b \geq B_i$ for all $1 \leq  i < b$.
%Similarly, the set of pseudo-necklaces of length $n$ with respect to $1^*0^*$ is a flip-swap language with respect to $0$.
%Thus, we can similarly list elements of the set so that each successive pseudo-necklaces differ by at most two bits cyclicly.

In~\cite{neck-sww}, the authors proved that the BRGC order induces a cyclic $2$-Gray code for the set of pseudo-necklaces of length $n$.

\subsection*{Left factors of $k$-ary Dyck words}
Recall that a \emph{$k$-ary Dyck word} is a binary string of length $n = tk$ with $t$ copies of $1$ and $t(k - 1)$ copies of $0$ such that every prefix has at most $k - 1$ copies of $0$ for every $1$.
It is well-known that $k$-ary Dyck words are in one-to-one correspondence with $k$-ary trees with $t$ internal nodes.
When $k = 2$, Dyck words are counted by the Catalan numbers and are equivalent to \emph{balanced parentheses}.
%A string is a \emph{balanced parentheses string} if it contains an equal number of open parentheses and closed parentheses, and every prefix of the string contains at least as many open parentheses as closed parentheses.
As an example, $110100$ is a $2$-ary Dyck word and is also a balanced parentheses string while $100110$ is not a $2$-ary Dyck word nor a balanced parentheses because its prefix of length three contains more $0$s than $1$s.
%A balanced parentheses string can be represented by a binary string using $0$ and $1$ to represent open parentheses and closed parentheses respectively.
%For example, the balanced parentheses string $(()())$ can be represented by the binary string $001011$.
$k$-ary Dyck words and balanced parentheses strings are well studied and have lots of applications including trees and stack-sortable permutations~\cite{Bultena:1998:EAW:306049.306064,Ruskey199068,Ruskey:2008:GBP:1379361.1379382,Vajnovszki:2006:LTG:1219189.1219688}.

The set of $k$-ary Dyck words of length $n$ is not a flip-swap language with respect to $0$ since $110100$ is a $2$-ary Dyck word but $111100$ is not.
The set of length $n$ prefixes of $k$-ary Dyck words is, however, a flip-swap language with respect to $0$.
This set is also called \emph{left factors of $k$-ary Dyck words}.
Let ${\bf S}$ be the set of left factors of $k$-ary Dyck words. %and $\alpha \neq 0^n$ be a string in ${\bf S}$.
%Consider a string $\alpha \neq 0^n$ in ${\bf S}$.
Observe that ${\bf S}$ satisfies the two closure properties of a flip-swap language with respect to $0$ as the flip-first and swap-first operations do not increase the number $0$s in the prefix.
Thus, ${\bf S}$ is a flip-swap language with respect to $0$.

\end{document}